\documentclass[12pt]{article}
\usepackage{setspace}
\usepackage{textcomp}
\usepackage[dvips]{color}
\usepackage{epsfig}
\usepackage{amsmath, amsthm, amssymb}
\usepackage[left=1.5in,top=.9in, 
right=1in,nohead]{geometry}
\usepackage{mathrsfs}
\usepackage{rawfonts}
\usepackage{graphicx}
\usepackage{comment}
\usepackage{float}
\usepackage{todonotes}
\usepackage[all]{xy}
\usepackage[sc,osf]{mathpazo}  
\usepackage{multirow}
\usepackage{pgf,tikz}
\usetikzlibrary{arrows,cd,matrix,decorations.pathmorphing}
\usetikzlibrary{babel}            
\usepackage{hyperref}

\usepackage{bmpsize}

\usepackage{comment}

\def\bct{\begin{center}}
\def\ect{\end{center}}
\def\bit{ \begin{itemize} }
\def\eit{ \end{itemize} }
\def\beg{\begin}

\def\<{\langle}
\def\>{\rangle}

\def\mbb{\mathbb}
\def\mbbcp{\mathbb{CP}}

\def\ni{\noindent}

\def\tcm{\textcolor{magenta}}
\def\tn{\textnormal}

\newtheorem{thm}{Theorem}[section]

\newtheorem{prop}[thm]{Proposition}

\newtheorem{rmk}[thm]{Remark}

\theoremstyle{plain}

\numberwithin{equation}{section}

\title
{Minimal submanifolds and stability in 
Einstein manifolds}
\author{Mustafa Kalafat \and \"{O}zg\"{u}r Kelek\c{c}i 
\and Mert Ta\c{s}demir
}

\begin{document}
\maketitle
\hfill  {\em To the memory of 
Jim Simons} 

\begin{abstract}
In this paper, we 
compute the index and nullity for minimal submanifolds 
of some complex 
Einstein spaces. We investigate the stability of these minimal submanifolds and suggest a criterion for instability for some cases. 
We also compute some higher eigenvalues for the Laplacian of the Berger spheres and provide with an algorithm. 


\vspace{.05in}

\ni {\em Keywords:} Einstein metrics;\,Hermitian metrics;\,
minimal submanifold;\,Laplacian. 

\vspace{.05in}

\ni {\em Mathematics Subject Classification 2010:} Primary 53A10
; Secondary 
35P05,53C25.  
\end{abstract}

\maketitle

\section{Introduction}

 Minimal submanifolds are an intriguing and fundamental concept in differential geometry, characterized by the property that they locally minimize area. These submanifolds are critical points for the area functional, meaning their mean curvature vector vanishes at every point. The study of minimal submanifolds has profound implications in various fields, including geometric analysis, mathematical physics, and general relativity.
In the 20th century, the field experienced significant advances. 
Simons' work bridged differential geometry and global analysis, significantly influencing the study of minimal submanifolds.
\vspace{2mm}
\par This paper investigates minimal submanifolds within specific Einstein spaces, notably 
the Page space as the  
blow-up of 
the complex projective plane $\mathbb{CP}_2$,
endowed with its Einstein-Hemitian metric. 
Our focus is on calculating the indices and nullities of the 
submanifolds, providing a deeper understanding of their geometric properties. The concept of index, introduced by 
\cite{simons68}, is defined as the 
sum of the dimensions of the eigenspaces corresponding to the negative eigenvalues  of the Jacobi operator. Simons also introduced the notion of nullity, defined as 
the dimension of the eigenspace corresponding to the zero eigenvalue of the Jacobi operator. 
In his seminal work, Simons computed these quantities for minimal submanifolds in round spheres, establishing the following key result.
\begin{thm}\cite{simons68} A totally geodesic sphere $\mbb S^p$ immersed in a round $\mbb S^n$ has 
$$\tn{index}=n-p \tn{ ~~and~~ nullity}=(n-p)(p+1).$$     
\end{thm}

\ni He also proceeds to prove the following lower bounds on the index and nullity of a general minimal submanifold. \vspace{2mm}
\begin{thm}\cite{simons68}\label{simonsunstabilityinsphere} Let $M$ be a $p$-dimensional, closed, minimal submanifold immersed in $\mbb S^n$. Then the following bounds hold. 
\begin{enumerate} \item The index of $M$ is greater than or equal to $(n-p)$, and the equality holds only when $M$ is $\mbb S^p$. 
\item The nullity of M is greater than or equal to $(n-p)(p+1)$, and the equality holds only when $M$ is $\mbb S^p$. 
\end{enumerate}  \end{thm}
\ni There exist further advancements on the minimal submanifolds of spheres  
in the late 20th century. 
Among them, building up on the foundations of Simons, Ejiri provided better estimates 
for the index of minimal $2$-spheres 
in higher-dimensional spheres. 
\begin{thm}\cite{ejiri83} Let M be a closed orientable minimal surface of genus zero 
immersed in $\mbb S^{2n}$ for $n>2$. Then the index of $M$ is greater than or equal to $2(n(n+2)-3)$.\end{thm}
\ni In this paper, we study the index and nullity of the Berger sphere on different complex Einstein manifolds, which have never been investigated in this aspect before. To achieve this, we introduce a theorem 
on the Jacobi operator of totally geodesic submanifolds in Einstein manifolds.  \vspace{2mm}

\noindent {\bf Theorem \ref{kalafat_theorem}.} 
{\em Let $f:M\to \bar{M}$ be a totally geodesic immersion of a $p$-manifold 
into an Einstein $n$-manifold. The following are true.
\beg{enumerate}
\item The Jacobi operator of $M$ is equal to  $J=-\Delta_M-(s/n)I+Ric^\perp $. 

\item If $M$ is a hypersurface, 
then  $J=-\Delta_M-(s/n) I$. 

\item If $\bar{M}$ is of constant curvature, 
then  $J=-\Delta_M-s\frac{p}{n(n-1)} I$. 
\end{enumerate}  } 
\ni Then we study the spectrum of the Laplacian on the related Riemannian spaces, and developed an idea to calculate the eigenvalues of the Jacobi operator. 
Along the way we also compute some higher eigenvalues for the Laplacian of the Berger spheres and provide with an algorithm. 
For the discussion and results on the complex projective space we refer to \cite{28fubinistudy}. 
\ni On the other hand on Page space we have index and nullity strictly depending on the value of the parameter of the Berger sphere. In particular they have a strictly positive index which means the Berger sphere behaves contrary to the compact Kähler submanifolds in terms of index. Next we are 
able to find the following result inside the blow up of the complex projective plane.  \vspace{3mm}

\noindent {\bf Theorem \ref{bergerinpageindexnullity}.} 
{\em The Index of the Berger sphere in the Page space is 
        \begin{equation*}
            \text{Index}(\mbb S^3_r)=\begin{cases}
                5 & \text{for }\, r_1\leq r \leq r_2 \\
                1 & \text{otherwise}
            \end{cases}
        \end{equation*}
        where $r_1$ , $r_2$ are the roots of the following function 
        $$2f^{-1}+U^2 D^{-2}\sin^{-2}r-3(1+a^2)$$ in the interval $(0,\pi)$ and the nullity is
        \begin{equation*}
            \text{Nullity}(\mbb S^3_r)=\begin{cases}
                4 & \text{for } r=r_1,r_2 \\ 
                0 & \text{otherwise}
            \end{cases}
        \end{equation*} }




\vspace{2mm} 
\ni At the end, we also give a criterion for instability of submanifolds. 
\vspace{2mm} 

\noindent {\bf Theorem \ref{instabilitycriterion}.} 
{\em Let $f:M\to \bar{M}$ be a totally geodesic immersion of a $p$-manifold 
into an Einstein $n$-manifold of positive scalar curvature. If $M$ is a hypersurface or $\bar{M}$ is of constant curvature then $M$ is unstable. 
}

\vspace{2mm} 
The paper is organized as
follows. In section $\S$\ref{secjacobi} we investigate the Jacobi operator for submanifolds in Einstein spaces, $\S$\ref{secspectrum} we review and also give new computations on the spectrum of the Laplacian, in 
$\S$\ref{secpagesubmanifolds}   
we give applications. 
%
Finally in $\S$\ref{secstabilityminimal} 
we review the basic facts related to the stability of minimal submanifolds in a Riemannian manifold. 
\vspace{2mm}

The present work builds upon our earlier papers \cite{confk,ehpbisec} and \cite{page}. Further developments along these lines are pursued in the subsequent works \cite{kkt2gravitationalinstantons,28fubinistudy}.


\vspace{2mm}

\noindent{\bf Acknowledgements.} We would like to thank Tommaso Pacini and Matthias Kreck for useful discussions, as well as   
\c{C}. 
Hac\i yusufo\u{g}lu and J. 
Madnick for their lectures on minimal surfaces. 
We dedicate this paper to the memory of {\em James H. Simons}, who passed away during the preparation of this work; his contributions to mathematics and finance, as well as his generous support of academia, had a profound impact on many, including the authors. The authors declare no conflict of interest.



\vspace{2mm} 

\section{The Jacobi Operator}\label{secjacobi}

Whenever we have an immersion $f:M\to \bar{M}$ of a $p$-manifold into an $n$-manifold, one can consider the Jacobi operator on the sections of the normal bundle, 

$$J:C^\infty(NM)\to C^\infty(NM)$$ is defined by $$J=-tr\nabla^2+\bar{R}-\tilde{S}.$$ 


\vspace{2mm}

\ni Here, the third term $\tilde{S}={^\mathsf{T}}S\circ S$ is the composition of the transpose of the shape operator with itself considered as a map 
$S: NM \to Sym^2(TM)$. 
Another important term of the Jacobi operator is the so called {\em a partial Ricci operator} by \cite{simons68} acting on the normal bundle which is defined through  
$$\bar{R}V:=\sum_{i=1}^p (\bar{R}_{E_i V}E_i)^N ~~\tn{for}~~ E_1\cdots E_p \in TM_m  ~~\tn{and}~~ V\in NM_m$$
is non-zero unless the underlying manifold is Ricci-flat.  
If the ambient space is {\em Einstein} i.e. $Ric=\lambda g$, then this operator admits further decomposition. Moreover, if the 
immersed submanifold is totally 
geodesic the Jacobi operator has vanishing third term and thus admits 
the following simpler form. 


\begin{thm}
\label{kalafat_theorem} 
Let $f:M\to \bar{M}$ be a totally geodesic immersion of a $p$-manifold 
into an Einstein $n$-manifold. The following are true.
\beg{enumerate}
\item The Jacobi operator of $M$ is equal to  $J=-\Delta_M-(s/n)I+Ric^\perp $. 

\item If $M$ is a hypersurface, 
then  $J=-\Delta_M-(s/n) I$. 

\item If $\bar{M}$ is of constant curvature, 
then  $J=-\Delta_M-s{ p \over n(n-1)} I$.

\end{enumerate}
\end{thm}

\begin{proof} In the tangent space we can pointwise complete to an orthonormal basis 
$$\{E_1\cdots E_p, E_{p+1}{=}V_1 \cdots E_n{=}V_{n-p}\}$$ of the ambient space  $\bar{M}$. In order to understand the middle term of the Jacobi operator, we take the inner product of its image with an arbitrary normal basis element.  
$${
\renewcommand{\arraystretch}{2} 
\begin{array}{rcl} 
\< \bar RV,V_j\>  
& = & \<\sum_{i=1}^p (\bar{R}_{E_i V}E_i)^N,V_j\>
      =\sum_{i=1}^p \< \bar{R}_{E_i V}E_i,V_j\>\\
& = & \sum_{i=1}^p \<\bar R_{E_i V}E_i,V_j\>
      \pm \sum_{i=p+1}^n \< \bar{R}_{E_i V}E_i,V_j\> \\
& 
= & -Ric(V,V_j) - \sum_{i=1}^{n-p} \< \bar R_{V_i V}V_i,V_j\>\\
& = & \<-Ric(V),V_j\>+\<Ric^\perp(V),V_j\>,
\end{array} }$$  
where
$$Ric^\perp : NM \to NM$$ 
is defined by $$V \longmapsto -\sum_{i=1}^{n-p} (\bar R_{V_i V}V_i)^N$$ 


\vspace{2mm}

\ni is the Ricci curvature of the normal bundle. So that the partial Ricci 
operator term is $\bar R=-Ric+Ric^\perp$. In the Einstein case the constant is $\lambda=s/n$ which is a multiple of the constant scalar curvature. This completes the first part. 

\vspace{2mm}

In the case of $p=n-1$ since the normal bundle is $1$-dimensional, 
Ricci curvature of the normal bundle vanishes. So that 
$$\bar R=-Ric=-\lambda I$$
then this operator acts as a multiple of the identity operator.  
In the stronger case of constant $k_0$ curvature, for the Einstein manifold, we have a special form of the curvature tensor  $\bar{R}_{XY}=k_0(Yi_X - Xi_Y)$. For $V\in NM_m$, 
$${\renewcommand{\arraystretch}{2} 
\begin{array}{rcl} 
Ric^\perp V 
& = & -k_0\sum_{i=1}^{n-p} (\<V,V_i\>V_i
    - \<V_i,V_i\>V)^N\\
& = & -k_0 (V-(n-p)V)\\ 
& = & -k_0 (1-n+p)) V\\ 
& = & s { n-p-1 \over n(n-1)} V. 

\end{array} }$$  


\vspace{2mm}

\ni Summation with the Ricci term yields the result. 
\end{proof}

\vspace{2mm}
 
\ni Along the way, we have also proved the following. 

\vspace{2mm}

\begin{prop} In the case of an immersion of a $p$-manifold 
into an Einstein $n$-manifold, 
the normal Ricci curvature vanishes if and only if the partial Ricci operator is of the form, $$\bar R=-(s/n) I$$ i.e. a multiple of the identity operator. 
\end{prop}

\vspace{2mm} 

\section{Spectrum of the Laplacian}\label{secspectrum}

In this section, we are going to review the spectrum of the Laplacian on some compact Riemannian manifolds. Some resources on the subject are  \cite{lafonten} and \cite{bergergauduchonspektrum} or more recent \cite{hajlasz}. 
We start with the unit round $p$-sphere $\mbb S^p$.
To obtain a positive operator thus a positive spectrum, 
$$-tr\nabla^2_{\mbb S^p} \, P=~k\,(k+p-1)~P$$

\ni where the Eigenfunctions in this case, constitute the space 
$E_k=H_k\mbb R^{p+1}$ of the harmonic $k$-homogenous polynomials 
restricted from the Euclidean $(p+1)$-space.
Computation of the dimension of this space depends on the following observation\cite{axlerharmonic}. 
A $k$-homogenous polynomial in $n$-space has a harmonic part plus a part coming from the isomorphic image of lower degree polynomials $\mathcal P_{k-2}^n$ 
under the multiplication map with radius squared.   
See the Table \ref{table:specSp} for the details.
\begin{table}[ht] \caption{ {\em Spectrum of the Laplacian of the round $p$-sphere.
}}
\bct{\Large {\renewcommand*{\arraystretch}{1.7}
\hspace{-.2cm}\resizebox{15.2cm}{!}{
$\begin{array}{|c|c|c|c|c|c|c|c|}\hline
k & 0 & 1 & 2 & 3 &\cdots&k&\cdots\\\hline
Spec {: \lambda_k} & ~0~ &p& ~2\,(p+1)~ & 3\,(p+2) &~\cdots~&~k\,(k+p-1)~&~\cdots~\\\hline
~dim \,E_{\lambda_k}~ 
&~~1~~&~p+1~& (p+1)p/2-1 &~(p+1)\,p\,(p-1)\,/\,3!-(p+1)~&~~\cdots~~&
C_{k+p,p}-C_{k+p-2,p}
&~~\cdots~~\\ \hline
\end{array}$}} }\ect
\label{table:specSp} \end{table}

\ni Next, we continue with the so-called {\em squashed sphere}  or {\em the Berger sphere}. In general, one can define 
these metrics on the odd-dimensional spheres $\mbb S^{2n+1}$. The sphere is considered as the total space of the circle bundle on the complex projective space $\mbbcp_n$ of first Chern class $-1$. Then, it is furnished by the pullback of the Fubini-Study metric plus the fiber metric which makes the totally geodesic fibers equal in diameter. 
As a consequence of the fact that $3$(and also $7$)-dimensional sphere is parallelizable, on $\mbb S^3$ we have a simple expression for the Berger metric, 
$$g_\epsilon:=\sigma_1^2+\sigma_2^2+\epsilon^2\sigma_3^2.$$
Here, $\epsilon$ is a positive constant and $\{\sigma_1, \sigma_2, \sigma_3\}$ are the standard left invariant 1-forms on the Lie group $SU_2\approx \mbb S^3$. These can be expressed in terms of the variables based on the standard embedding of the 3-sphere into 4-space as follows,
$$\beg{array}{l}
r^2\sigma_1= x_1dx_0-x_0dx_1+x_2dx_3-x_3dx_2
\nonumber\\ [2\jot] 
r^2\sigma_2= x_2dx_0-x_0dx_2+x_3dx_1-x_1dx_3
\\ [2\jot] 
r^2\sigma_3= x_3dx_0-x_0dx_3+x_1dx_2-x_2dx_1
.\nonumber \label{relation5} 
\end{array}$$

\ni\cite{urakawa} is the first to compute the first nonzero eigenvalue 
for the 3-dimensional Berger spheres a particular case of the general setting of compact Lie groups. Later on, 
\cite{tanno}  
computes that for the Berger spheres in all odd dimensions which have the same volume as the round sphere of the dimension, which we call unit volume for simplicity. Restricting his results to the dimension three one uses the family of unit volume Berger metrics on $\mbb S^3$, 
$$g_B^t=t^{-1}g_1+(t^2-t^{-1})\,\sigma_3^2 ~~~\tn{for}~~ t\in \mbb R^+.$$

\ni Tanno computes the smallest positive eigenvalue for this case as follows.  

$$\lambda_1(g_B^t) =\left\{\begin{array}{cc}
 8t            & \tn{for}~~~ t \leq 1/\sqrt[3] 6  \\ 
t(2+t^{-3})& \tn{for}~~~ t \geq 1/\sqrt[3] 6 
\end{array}   \right.$$

\ni Normalizing the volume of the $\epsilon$-Berger metric we get 
$g_\epsilon^u=\epsilon^{-2/3}g_\epsilon$ which corresponds to the $t=\epsilon^{2/3}$ case of Tanno's. So that we obtain the following identification,  
$$g_\epsilon=\epsilon^{2/3} \,g^{\epsilon^{2/3}}_B.$$

\ni To be able to apply Tanno, we need one more step. Laplace-Beltrami operator on functions in local coordinates reveals the fact that it anti-transforms under the dilation of the metric, i.e. $\Delta_{\lambda g}f=\lambda^{-1}\Delta_{g}f$. 
This result is true in the case of constant $\lambda$ or if the space is a surface. To prove the latter assertion one needs to work out the general conformal scaling formula though. In the case that it holds, although the eigenfunctions stay the same, the eigenvalues are anti-multiplied by $\lambda$.  
Inserting the $t$ value and anti-multiplying we finally arrive at the first eigenvalue. 
$$\lambda_1(g_\epsilon) =\left\{\begin{array}{cc}
 8            & \tn{for}~~~ \epsilon \leq 1/\sqrt 6  \\ 
2+1/\epsilon^2& \tn{for}~~~ \epsilon \geq 1/\sqrt 6 
\end{array}   \right.$$

\vspace{2mm}

\ni 
        One can find further eigenvalues by using the following method. We need to sort eigenvalues in ascending order. To do that, for each step, we determine a minimum. To proceed further we take $m=3$, as we are interested in three-dimensional Berger spheres we consider only this dimension. As in the proof above, the minimums come from the following three cases. 
        \begin{enumerate}
            \item $tn(n+2) - t(1 - t^{-3})n^2$, \quad $n \geq 1$,
            \item $tk(k+2) - t(1 - t^{-3})$, \quad $k$ is odd and $\geq 1$,
            \item $tl(l+2)$, \quad $l$ is even and $\geq 2$.
        \end{enumerate}
        If we rearrange them we get the following curves. 
        \begin{enumerate}
            \item $\gamma_n(t)=tn(2+nt^{-3})$, \quad $n \geq 1$,
            \item $\alpha_k(t)=tk(k^2+2k-1+t^{-3})$, \quad $k$ is odd and $\geq 1$,
            \item $\beta_l(t)=tl(l+2)$, \quad $l$ is even and $\geq 2$.
        \end{enumerate}
         The first nonzero eigenvalue was found by comparing the minimums of those three and formed by splitting cases. In terms of the new notations, we can write the first eigenvalue as
         \begin{equation*}
            \lambda_1(g(t)) = 
            \begin{cases} 
                \alpha_1(t)=\gamma_1(t) & \text{for } t^{-3} \leq 6 \\
                \beta_2(t) & \text{for } t^{-3} \geq 6 
            \end{cases}.
        \end{equation*}
        To find $\lambda_2(g(t))$, we need to compare other possible curves, those are $\beta_2,\gamma_2,\alpha_3$, since $\alpha_3$ is far bigger than others, after the comparisons we have
        \begin{equation*}
            \lambda_2(g(t)) = 
            \begin{cases} 
                \gamma_2(t) & \text{for } t^{-3} \leq 1 \\
                \beta_2(t) & \text{for } t^{-3} \geq 1 
            \end{cases}.
        \end{equation*}
        If we proceed further we realize that we first start with $\beta_2(t)$, then pass to possible $\alpha_k$ which is $\alpha_1$, then pass possible $\gamma_n$, which is $\gamma_1$ but since they are equal we take both at the same time but this never be the case again. Then for the next eigenvalue, we start with $\beta_2,$ then try to pass $\alpha_3$ but $\alpha_3>\beta_2$ hence we move with $\gamma_2$. For the next eigenvalue, again the same logic: start with $\beta_2$ then move with $\gamma_3$. We proceed with this until no $\gamma$ under $\beta_2$ is left. Whenever there is no $\gamma$ left under $\beta_2$ we take $\beta_2$ for all $t$ as eigenvalue. The next step moves with $\beta_4$, then we try $\alpha_3$ as next, and other $\gamma$'s. We move with $\beta_4$ and $\alpha_3$ until no $\gamma$ left under them, then we pass two cased eigenvalue $\beta_4$ and $\alpha_3$, then we pass to the cases $\beta_4$ and other $\gamma$'s lie between $\alpha_3$ and $\beta_4$. Then we pass $\beta_4$ as eigenvalue itself. This process continues like this. 
        Roughly speaking we start our eigenvalue with $\beta_l$ as the first case then take $\alpha_{l-1}$ as the second case, and then take suitable $\gamma$ as the third case. We increase the index of $\gamma$ until they become greater than $\alpha_{l-1}$. Whenever they do, we take the next eigenvalue in the form
        \begin{equation*}
            \lambda_p(g(t)) = 
            \begin{cases} 
                \alpha_{l-1}(t) & \text{for } t^{-3} \leq A \\
                \beta_l(t) & \text{for } t^{-3} \geq A 
            \end{cases}.
        \end{equation*}
        The next few eigenvalues will have form
        \begin{equation*}
            \lambda_q(g(t)) = 
            \begin{cases} 
                \gamma_n(t) & \text{for } t^{-3} \leq A \\
                \beta_l(t) & \text{for } t^{-3} \geq A 
            \end{cases}
        \end{equation*}
        until $\gamma_n$'s start to become greater then $\beta_l$. Whenever they do we select the next eigenvalue as $\beta_l$ itself. Then we increase the index of $\beta$ and move in this algorithm. By using this algorithm, we obtained the first eleven eigenvalues as follows.
        \begin{equation*}
            \lambda_1(g(t)) = 
            \begin{cases} 
                t(2+t^{-3}) & \text{for } t^{-3} \leq 6 \\
                8t & \text{for } t^{-3} \geq 6 
            \end{cases},
        \end{equation*}
        \begin{equation*}
            \lambda_2(g(t)) = 
            \begin{cases} 
                2t(2+2t^{-3})) & \text{for } t^{-3} \leq 1 \\
                8t & \text{for } t^{-3} \geq 1
            \end{cases},
        \end{equation*}
        \begin{equation*}
            \lambda_3(g(t)) = 
            \begin{cases} 
                3t(2+3t^{-3}) & \text{for } t^{-3} \leq \frac{2}{9} \\
                8t & \text{for } t^{-3} \geq \frac{2}{9}
            \end{cases},
        \end{equation*}
        \begin{equation*}
            \lambda_4(g(t))=8t
        \end{equation*}
        \begin{equation*}
            \lambda_5(g(t)) = 
            \begin{cases} 
                4t(2+4t^{-3}) & \text{for } t^{-3}\leq\frac{2}{5}\\
                t(14+t^{-3}) & \text{for } \frac{2}{5}\leq t^{-3} \leq 10 \\
                24t & \text{for } t^{-3} \geq 10
            \end{cases},
        \end{equation*}
        \begin{equation*}
            \lambda_6(g(t)) = 
            \begin{cases} 
                5t(2+5t^{-3}) & \text{for } t^{-3}\leq\frac{1}{6}\\
                t(14+t^{-3}) & \text{for } \frac{1}{6}\leq t^{-3} \leq 10 \\
                24t & \text{for } t^{-3} \geq 10
            \end{cases},
        \end{equation*}
        \begin{equation*}
            \lambda_7(g(t)) = 
            \begin{cases} 
                6t(2+6t^{-3}) & \text{for } t^{-3}\leq\frac{2}{35}\\
                t(14+t^{-3}) & \text{for } \frac{2}{35}\leq t^{-3} \leq 10 \\
                24t & \text{for } t^{-3} \geq 10
            \end{cases},
        \end{equation*}
        \begin{equation*}
            \lambda_8(g(t)) = 
            \begin{cases} 
                t(14+t^{-3}) & \text{for } t^{-3} \leq 10 \\
                24t & \text{for } t^{-3} \geq 10
            \end{cases},
        \end{equation*}
        \begin{equation*}
            \lambda_9(g(t)) = 
            \begin{cases} 
                7t(2+7t^{-3}) & \text{for } t^{-3} \leq \frac{10}{49} \\
                24t & \text{for } t^{-3} \geq \frac{10}{49}
            \end{cases},
        \end{equation*}
        \begin{equation*}
            \lambda_{10}(g(t)) = 
            \begin{cases} 
                8t(2+8t^{-3}) & \text{for } t^{-3} \leq \frac{1}{8} \\
                24t & \text{for } t^{-3} \geq\frac{1}{8}
            \end{cases},
        \end{equation*}
        \begin{equation*}
            \lambda_{11}(g(t)) = 
            \begin{cases} 
                9t(2+9t^{-3}) & \text{for } t^{-3} \leq \frac{2}{27} \\
                24t & \text{for } t^{-3} \geq\frac{2}{27}
            \end{cases}.
        \end{equation*}
        \begin{figure}[ht]
            \centering
            \includegraphics[scale=0.6]{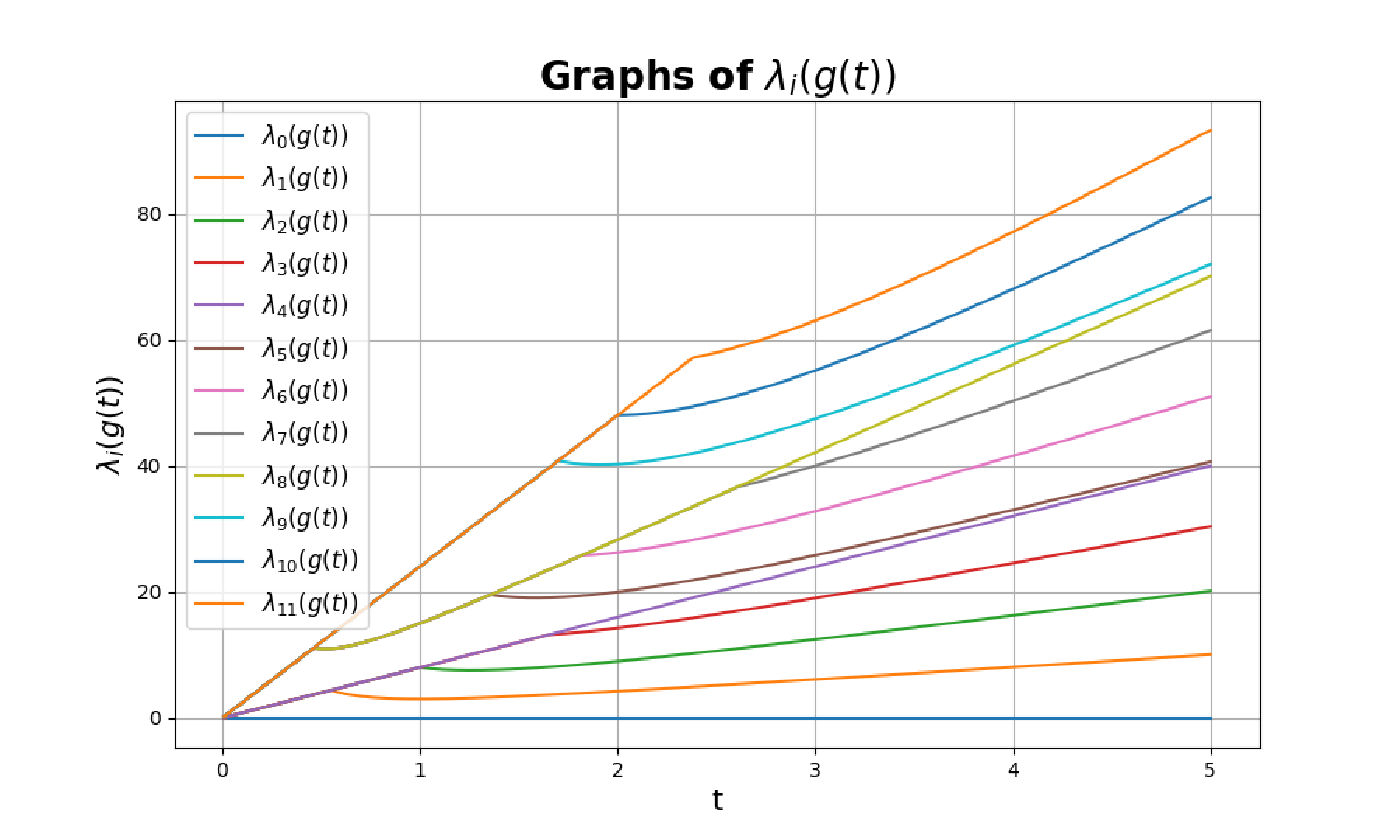}
            \caption{
            Eigenvalues for the Berger sphere Laplacians}
            \label{fig:graphs-of-lambda}
        \end{figure}

\vspace{1cm}

\begin{rmk} Other spectrum computations for Riemannian manifolds include complex projective spaces and tori \cite{berard-bergerybourguignonbergerlaplacaian}. 
See also \cite{Renato} for general Berger spheres.  
\end{rmk}


\section{Submanifolds of the Page Space}\label{secpagesubmanifolds}

    In this section, 
    we use the Theorem \ref{kalafat_theorem}, and Tanno's idea we reach out to the index of the Berger sphere in the Page space. First, the Berger sphere in Page space can be given by fixing the $r$ coordinate in Page metric. The resulting metric is given by
    
    \begin{equation*}
        g_{\text{BS}}=f(\sigma_1^2+\sigma_2^2)+\frac{C\sin^2r}{V}\sigma_3^2
    \end{equation*}

    \vspace{2mm}

    \ni which is totally geodesic in Page space by \cite{page}. Here, the possible values for $r$ is the interval $(0,\pi)$.
    \begin{thm}\label{bergerinpageindexnullity}
        The Index of the Berger sphere in the Page space is 
        \begin{equation*}
            \text{Index}(\mbb S^3_r)=\begin{cases}
                5 & \text{for }\, r_1\leq r \leq r_2 \\
                1 & \text{otherwise}
            \end{cases}
        \end{equation*}
        where $r_1$ , $r_2$ are the roots of the following function 
        $$2f^{-1}+U^2 D^{-2}\sin^{-2}r-3(1+a^2)$$ in the interval $(0,\pi)$ and the nullity is
        \begin{equation*}
            \text{Nullity}(\mbb S^3_r)=\begin{cases}
                4 & \text{for } r=r_1,r_2 \\ 
                0 & \text{otherwise}
            \end{cases}
        \end{equation*}
    \end{thm}  

    \begin{proof}
        As in the proof for $\mbbcp_2$ we first normalize the Berger sphere. The volume for the Berger sphere is
        \begin{equation*}
            Vol(g_{BS})=\frac{fD\sin r}{U}Vol(\bar{g})
        \end{equation*}
        where $\bar{g}$ is the metric for the Berger sphere in the unit sphere $\mbb S^3$ with the volume same as the volume of $\mbb S^3$. 
        So that we have a scaling factor 
        $\mu=\left({fD\sin r}/{U}\right)^{2/3}$ 
        to write as follows. 
        \begin{equation*}
            g_{\text{BS}}=\left(\frac{fD\sin r}{U}\right)^{2/3}\Bar{g}. 
        \end{equation*}
        We can give the metric $\bar{g}$ explicitly as
        \begin{equation*}
            \bar{g}=\left(\frac{D\sin r}{U}\right)^{-2/3}f^{1/3}(\sigma_1^2+\sigma_2^2)+\left(\frac{D\sin r}{U}\right)^{4/3}f^{-2/3}\sigma_3^2.
        \end{equation*}
         If we compare the metrics above and in the convention \ref{tanno-last-metric} of a Berger sphere, 
         
        \begin{equation}\label{tanno-last-metric}
            t^{-1}(\sigma_1^2+\sigma_2^2)+t^2\sigma_3^2.
        \end{equation}
we have the relationship between $t$ and $r$ as $t=U^{-2/3}\left({D\sin r}\right)^{2/3}f^{-1/3}$. Hence we can give the first nonzero eigenvalue for the Laplacian for $\bar{g}$ as
        \begin{equation*}
            \Bar{\lambda}_1(r)=\begin{cases} 
                (2+\left(\frac{D\sin r}{U}\right)^{-2}f)\left(\frac{D\sin r}{U}\right)^{2/3}f^{-1/3} & \text{for }\left(\frac{D\sin r}{U}\right)^{-2}f \leq \sqrt{6} \\
                8\left(\frac{D\sin r}{U}\right)^{2/3}f^{-1/3} & \text{for }\left(\frac{D\sin r}{U}\right)^{-2}f\geq \sqrt{6}
            \end{cases}
        \end{equation*}
        to calculate the index of the Jacobi operator, we need to calculate the first nonzero eigenvalue for $g_{\text{BS}}$. To do that we need to multiply the first eigenvalue for the Laplacian of $\Bar{g}$ by $\mu^{-1}=U^{2/3}\left({fD\sin r }\right)^{-2/3}$. Now we have the first eigenvalue for $g_{\text{BS}}$.
        For the index of the Jacobi operator, by Theorem \ref{kalafat_theorem}, we need to find the dimensions of eigenspaces of negative eigenvalues of $J=\Delta_{BS}-{s}/{4}\, I$, by \cite{page}, we know the scalar curvature for Page space is $12(1+ a^2) \simeq 12.952.$ If we apply those, we see that the first eigenvalue for the Jacobi operator is $\lambda_0(r)-3(1+a^2)$, and the second eigenvalue is
        \begin{equation*}
            \lambda_1(r)-3(1+ a^2)=
                \begin{cases} 
                \left(2+\left(\frac{D\sin r}{U}\right)^{-2}f\right)f^{-1}-3(1+a^2) & \text{for }\left(\frac{D\sin r}{U}\right)^{-2}f \leq \sqrt{6} \\
                8f^{-1}-3(1+a^2) & \text{for }\left(\frac{D\sin r}{U}\right)^{-2}f\geq \sqrt{6}
            \end{cases}
        \end{equation*}
        If we calculate the roots for $\lambda_1(r)-3(1+a^2)$ by using bisector method with tolerance $10^{-6}$, we get the roots $r_1\simeq 0.7032761573791504$ and $r_2=2.4383171081542976$. Hence the index and nullity for the interval $(0,r_1)$ and $(r_2,\pi)$ are $1$ and zero, respectively as the dimension of the eigenspace for $\lambda_0(r)$ is $1$, however for $(r_1,r_2)$ the index is greater than the other two intervals. If we sketch the other eigenvalues for the Jacobi operator we get the following graph.
        \begin{figure}[ht]
            \centering
            \includegraphics[scale=0.6]{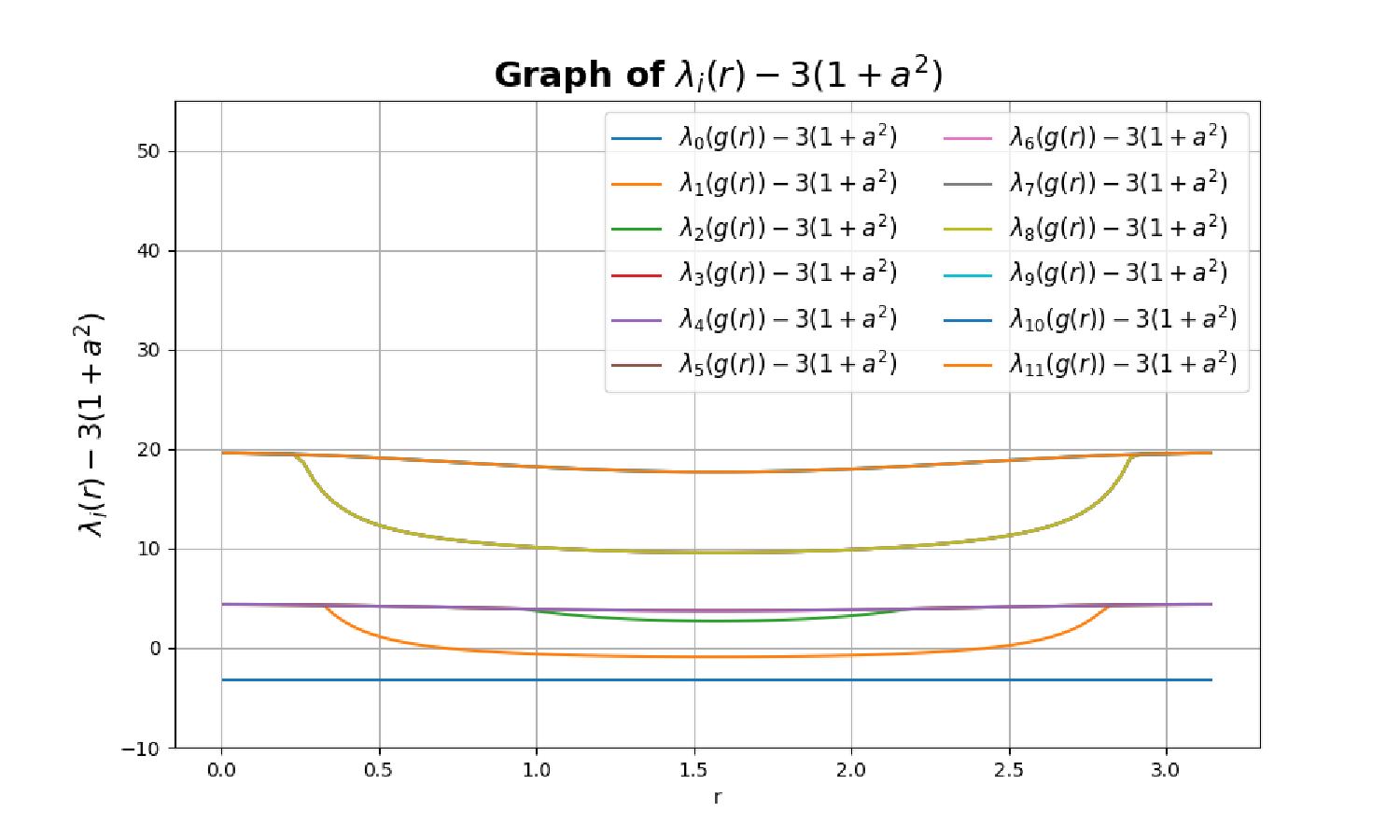}
            \caption{
            Eigenvalues of the Jacobi operator for the Berger spheres in Page space}
            \label{fig:lambda-page}
        \end{figure}
        Some eigenvalues are overlapped due to the interval for the cases. One can see that nullity is nonzero only for two values of $r$, and index is $5$ for $(r_1,r_2)$ as the dimension for the eigenspace of $\lambda_1(r)$ is $m+1=4$ by \cite{tanno}. Hence the index on this interval is $5$, on the outer part is $1$.
            \end{proof}

\vspace{2mm} 

\section{Stability of minimal submanifolds}\label{secstabilityminimal}

The second variational formula for the volume in terms of compactly supported 
normal variations is computed to be \cite{peterlibook}, 

\vspace{2mm}

$$\hspace{-10mm}\frac{d^2}{dt^2}V(N_t)|_{t=0} = \int_N\{ -\sum_{i,j}\langle T,II_{i,j}\rangle^2 
-\sum_{i=1}^n\langle \mathcal R_{e_iT}T,e_i\rangle + \langle \nabla_T^nT,H\rangle $$

$$+ \sum_{i=1}^n\sum_{\alpha=n+1}^m\langle \nabla_{e_i}T,e_\alpha\rangle + \langle T,H\rangle ^2 \}.$$

\vspace{6mm}

\ni Since complex submanifolds of K\"{a}hler manifolds are also K\"{a}hler, by the theorem of \cite{simons68} complex submanifolds of the complex projective space has zero index. There exist no negative eigenvalue for the Jacobi operator. So that there are no directions in the normal bundle to deform are to make area smaller. So that deformations of them are initially area increasing which we call {\em stable}. This is contrary to the case of the round spheres living inside round spheres. The Theorem \ref{simonsunstabilityinsphere} of Simons that we have cited in the introduction implies that all minimal submanifolds of the round sphere is unstable. A criterion can be given for instability using results from section 
$\S$\ref{secjacobi}.  
\begin{thm}\label{instabilitycriterion}   
Let $f:M\to \bar{M}$ be a totally geodesic immersion of a $p$-manifold 
into an Einstein $n$-manifold of positive scalar curvature. If $M$ is a hypersurface or $\bar{M}$ is of constant curvature then $M$ is unstable. 
\end{thm}
\begin{proof} This falls into the second case of the Theorem \ref{kalafat_theorem}. Since the Laplacian has nonnegative (and discrete) eigenvalues, smallest one is zero with multiplicity one, the dimension of the constant functions. This yields a negative eigenvalue for the Jacobi operator after shifting by the scalar curvature. Consequently the index is greater than or equal to one. \end{proof}

\ni As an application of this corollary, we can say that the totally geodesic hypersurfaces of $\mathbb{CP}_2$ and the Page space are unstable, in particular the Berger spheres that we are interested in this paper. 
Our direct computation in section 
$\S$\ref{secpagesubmanifolds} is an alternative verification of this result. 
The minimal submanifolds in this class behave more like the ones in spheres in terms of this respect rather than complex submanifolds. 

\appendix



\hspace{1cm}\vspace{7mm}

\bibliographystyle{alphaurl}
\bibliography{minimal}


\vspace{.05in}




\vspace{3mm}

{\small
\begin{flushleft}
\textsc{Mathematisches Institut, 
Rheinische Friedrich-Wilhelms-Universität Bonn
Endenicher Allee 60, 
D-53115 Bonn, 
Germany
.}\\ 
\textit{E-mail address:} 
\texttt{\textbf{kalafat@\,math.uni-bonn.de}}

\end{flushleft}
}

\vspace{1.5mm}

{\small 
\begin{flushleft} \textsc{Department of Basic Sciences \& Faculty of Engineering, \\ University of Turkish Aeronautical Association,  Ankara,  Turkiye
}\\
\textit{E-mail address:} \texttt{\textbf{okelekci@\,thk.edu.tr}}
\end{flushleft}
}

\vspace{1.5mm}

{\small
\begin{flushleft}
\textsc{Mathematisches Institut, 
Rheinische Friedrich-Wilhelms-Universität Bonn
Endenicher Allee 60, 
D-53115 Bonn, 
Germany
.}\\ 
\textit{E-mail address:} 
\texttt{\textbf{merttsdmr@\,uni-bonn.de}}

\end{flushleft}
}

\end{document}